\documentclass[oneside,english]{amsart}
\usepackage[T1]{fontenc}
\usepackage[latin9]{inputenc}
\usepackage{geometry}
\usepackage{babel}
\usepackage{amstext}
\usepackage{amsthm}
\usepackage{amssymb}
\usepackage[unicode=true,pdfusetitle,
 bookmarks=true,bookmarksnumbered=false,bookmarksopen=false,
 breaklinks=false,pdfborder={0 0 1},backref=false,colorlinks=false]
 {hyperref}

\makeatletter
\theoremstyle{plain}
\newtheorem{thm}{\protect\theoremname}
\theoremstyle{plain}
\newtheorem{cor}[thm]{\protect\corollaryname}

\makeatother

\providecommand{\corollaryname}{Corollary}
\providecommand{\theoremname}{Theorem}
\begin{document}

\title{a note on Clebsch-Gordan integral, Fourier-Legendre expansions and
closed form for hypergeometric series}

\author{Marco Cantarini}

\address{Dipartimento di Ingegneria Industriale e Scienze Matematiche \\
Universit\`a Politecnica delle Marche\\
Via Brecce Bianche, 12\\
60131 Ancona, Italia}

\email{m.cantarini@univpm.it}
\begin{abstract}
In this paper we show that a closed form formula for the generalized
Clebsch-Gordan integral and the Fourier-Legendre expansion theory
allow to evaluate hypergeometric series involving powers of the normalized
central binomial coefficient ${\displaystyle \frac{1}{4^{n}}\dbinom{2n}{n}}$.
\end{abstract}

\keywords{Hypergeometric functions, Fourier-Legendre expansion, Clebsch-Gordan
integral, complete elliptic integral of the first kind, closed form.}

\maketitle
Mathematical Subject Classification 2020: 33C20, 33E05, 42C10, 33C75.

\section{introduction}

The study of hypergeometric transformation, and its link to the analysis
of closed-form of infinite series in terms of well-known mathematical
constants and special values of Euler's Gamma function, has been deeply
analyzed in many articles and with different techniques. Indeed, it
is well known that this type of research is of interest mathematics
and in other scientific fields; a very exhaustive illustration can
be found in \cite{Bor}. Among the many tools developed, recently
it was shown that the Fourier-Legendre (FL) expansion theory is a
very useful approach for the study of a class of hypergeometric series,
in particular series whose summands are powers of the normalized central
binomial coefficients, harmonic numbers and rational functions (see
\cite{CamDauSon},\cite{CanDau}) because it allows to formulate these
series in terms of Euler sums or in integrals involving special functions
like polylogarithms or complete elliptic integrals of the first and
second kind. This information shows again the interest about this
type of problems since, as we know, the evaluation of multiple elliptic
integrals where the integrands are combinations of complete elliptic
integrals of the first or second kind is an active research area and
with applications in, for example, high-energy physics, statistical
mechanics and probability theory. Very recently \cite{CamCanDau}
it was observed that the FL theory combined with the theory of fractional
operators, in particular with semi-differentiation and semi-integration
(for some details about fractional calculus see, for example, \cite{KibMarSam})
of scalar product of some functions, including complete elliptic integrals
of the first or second kind, allow us to evaluate particular hypergeometric
functions with fractional (in particular, quarter-integers) parameters.

It is important to emphasize again how these topics and techniques
can be relate to other, and sometimes unexpected, mathematical topics;
for example, if it is quite natural to think about the classical Ramanujan-type
series for $1/\pi$ (for a survey of this topic see for example, \cite{BarBerCha}
and for formulas via hypergeometric transformations see \cite{CooGeYe}),
the connection with additive number theory problems is probably less
evident, in particular asymptotic formulas of functions that count
the number of representations of an integer as the sum of elements
that are in some subset of natural numbers (essentially, primes or
powers of primes). Indeed, fractional operators applied to particular
power series are involved in the study of explicit formulas for the
so-called Ces\`aro average of these counting functions (for the interested
reader, see \cite{BruKacPer}\cite{Can1}\cite{Can2}\cite{Can3}\cite{LanZac}),
therefore it is plausible to think that the techniques developed may
also be of interest for these types of problems

In this paper, we will focus on a results of Zhou \cite{zhou} about
a closed form for the generalized Clebsch-Gordan integral
\[
\int_{-1}^{1}P_{\mu}\left(x\right)P_{\nu}\left(x\right)P_{\nu}\left(-x\right)dx
\]
where $P_{\nu}\left(x\right),P_{\mu}\left(x\right)$ are the Legendre
functions of arbitrary complex degree $\nu,\mu\in\mathbb{C}$. We
show that this result can be can be interpreted in terms of the FL
theory and this point of view allows to evaluate series whose addends
are powers of central binomials (and so, particular hypergoemetric
functions). Furthermore, we will show that from Zhou results we can
obtain some formulas that recall the well-known Brafman's formula
\cite{Braf} and we can evaluate very easily some integral moment
regarding combinations of complete elliptic integrals of the first
kind.

Note that we adopt the convention whereby the argument of a complete
elliptic integral is the elliptic modulus, that is
\[
K(x):=\int_{0}^{\pi/2}\frac{du}{\sqrt{1-x\sin^{2}\left(u\right)}}.
\]
\,
I thank the referee very much for the comments and suggestions, which have greatly improved the article.

\section{On some consequences of Zhou's Paper}

We start our analysis observing that the closed form for the Clebsch-Gordan
integral can be interpreted as the FL expansions of a combination
of particular Gauss hypergeometric functions $_{2}F_{1}(a,b;c;z)$.
\begin{thm}
Let $\nu\in\mathbb{C}$ and $x\in\left[0,1\right]$. The following
FL expansions there holds:
\begin{equation}
_{2}F_{1}\left(-\nu,\nu+1;1;x\right){}_{2}F_{1}\left(-\nu,\nu+1;1;1-x\right)\label{eq:similbraf-1}
\end{equation}

\[
=-\frac{\sin\left(\pi\nu\right)}{2}\sum_{m\geq0}\left[\frac{1}{4^{m}}\dbinom{2m}{m}\right]^{2}\frac{\Gamma\left(m-\nu\right)\Gamma\left(m+\nu+1\right)\left(4m+1\right)}{\Gamma\left(m-\nu+\frac{1}{2}\right)\Gamma\left(m+\nu+\frac{3}{2}\right)}P_{2m}\left(2x-1\right)
\]
where the indeterminate form must be interpreted as limits.
\end{thm}

\begin{proof}
Recalling that
\[
_{2}F_{1}\left(-\nu,\nu+1;1;x\right)=P_{\nu}\left(1-2x\right),\,x\in\left[0,1\right],\,\nu\in\mathbb{C},
\]

(see equation $6.2$ of \cite{Kris}) we have, for $m\in\mathbb{N}$,
\begin{align}
&\int_{0}^{1}{}_{2}F_{1}\left(-\nu,\nu+1;1;x\right){}_{2}F_{1}\left(-\nu,\nu+1;1;1-x\right)P_{m}\left(2x-1\right)dx\nonumber \\  &= \int_{0}^{1}P_{\nu}\left(1-2x\right)P_{\nu}\left(2x-1\right)P_{m}\left(2x-1\right)dx\nonumber \\
 &= \int_{-1}^{1}P_{\nu}\left(x\right)P_{\nu}\left(-x\right)P_{m}\left(x\right)dx.\label{eq:zhou1}
\end{align}
From equation ($19_{(m,n)}$) and ($19_{(2m+1,\nu)}$) of \cite{zhou}
we have that (\ref{eq:zhou1}) is $0$ for any Legendre polynomial $P_{n}\left(2x-1\right)$ of odd degree $n$
and
\[
\int_{-1}^{1}P_{\nu}\left(x\right)P_{\nu}\left(-x\right)P_{2m}\left(x\right)dx=-\frac{\sin\left(\pi\nu\right)}{2}\left[\frac{1}{4^{m}}\dbinom{2m}{m}\right]^{2}\frac{\Gamma\left(m-\nu\right)\Gamma\left(m+\nu+1\right)}{\Gamma\left(m-\nu+\frac{1}{2}\right)\Gamma\left(m+\nu+\frac{3}{2}\right)}
\]
where $m$ is a positive integer, $\nu\in\mathbb{C}$ and the indeterminate
form must be interpreted as limits. The thesis follows recalling that
if $f(x)/\sqrt[4]{1-x^{2}},\,x\in(-1,1)$ is integrable, then
\[
\sum_{n\geq0}\left(n+\frac{1}{2}\right)P_{n}\left(\xi\right)\int_{-1}^{1}f\left(x\right)P_{n}\left(x\right)dx=\frac{f\left(\xi+0\right)-f\left(\xi-0\right)}{2}
\]

for a certain $\xi\in(-1,1)$ if some conditions for convergence are met (for more details see \cite{Hob}, Chapter VII, p. $329$).

\end{proof}
As we had anticipated, the previous formula clearly recalls the well-known
Brafman's formula (see \cite{Braf}) and it has some interesting consequences.
\begin{thm}
For $\nu\in\mathbb{C}\setminus\left(\left\{ -2\mathbb{N}+1\right\} \cup\left\{ 2\mathbb{N}\right\} \right)$
we have 
\begin{equation}
\frac{\cot\left(\frac{\pi\nu}{2}\right)\Gamma\left(\frac{1+\nu}{2}\right)^{2}}{\pi\Gamma\left(\frac{2+\nu}{2}\right)^{2}}=\sum_{m\geq0}\left[\frac{1}{4^{m}}\dbinom{2m}{m}\right]^{3}\left(-1\right)^{m+1}\frac{\Gamma\left(m-\nu\right)\Gamma\left(m+\nu+1\right)\left(4m+1\right)}{\Gamma\left(m-\nu+\frac{1}{2}\right)\Gamma\left(m+\nu+\frac{3}{2}\right)}\label{eq:cons1}
\end{equation}
where the undetermined forms must be interpreted as limits and, for
$x\in\left[0,1\right]$, we have
\begin{equation}
K\left(x\right)K\left(1-x\right)=\frac{\pi^{3}}{8}\sum_{m\geq0}\left[\frac{1}{4^{m}}\dbinom{2m}{m}\right]^{4}\left(4m+1\right)P_{2m}\left(2x-1\right).\label{eq:cons2}
\end{equation}
\end{thm}

\begin{proof}
Due to the fact that
\begin{equation}
_{2}F_{1}\left(-\nu,\nu+1;1;\frac{1}{2}\right)=\frac{\sqrt{\pi}}{\Gamma\left(\frac{1-\nu}{2}\right)\Gamma\left(\frac{\nu+2}{2}\right)}\label{eq:2f1 special}
\end{equation}
by the Gauss's second summation theorem \cite{Bai}, formula (\ref{eq:cons1})
follows from the $x=1/2$ case of (\ref{eq:similbraf-1}) and (\ref{eq:cons2})
follows from the $\nu=-1/2$ case of (\ref{eq:similbraf-1}).
\end{proof}
These results produce some interesting identities that CAS like Mathematica
does not recognize or it only recognizes them as combinations of generalized hypergeometric functions; in the next corollary we show some interesting examples. Note that we will write the combination of hypergeometric functions only in cases where the CAS is able to recognize it and if the formula contains at most two terms, to make the results more readable.

\begin{cor}
We have

\begin{align}
\sum_{m\geq0}\left[\frac{1}{4^{m}}\dbinom{2m}{m}\right]^{3}\frac{\left(-1\right)^{m+1}(4m+1)^{2}}{\left(4m-1\right)\left(4m+3\right)} = \frac{32\left(2+\sqrt{2}\right)\Gamma\left(\frac{1}{4}\right)^{2}}{\Gamma\left(\frac{1}{8}\right)^{4}} \label{eq:cor1}
\end{align}

\begin{equation}
\sum_{m\geq0}\left[\frac{1}{4^{m}}\dbinom{2m}{m}\right]^{3}\frac{\left(-1\right)^{m+1}\left(4m-1\right)\left(4m+3\right)}{\left(4m-3\right)\left(4m+5\right)}=\frac{32\sqrt{2}\left(1+\sqrt{2}\right)\Gamma\left(\frac{1}{4}\right)^{2}}{9\Gamma\left(\frac{1}{8}\right)^{4}},\label{eq:cor1.5}
\end{equation}
\begin{equation}
\sum_{m\geq0}\left[\frac{1}{4^{m}}\dbinom{2m}{m}\right]^{5}\frac{\left(-1\right)^{m}\left(4m+1\right)\left(4m^{2}+2m+1\right)}{\left(2m-1\right)^{2}\left(m+1\right)^{2}}=\frac{128}{\Gamma\left(\frac{1}{4}\right)^{4}},\label{eq:cor2}
\end{equation}
\begin{equation}
\sum_{m\geq0}\left[\frac{1}{4^{m}}\dbinom{2m}{m}\right]^{5}\left(-1\right)^{m}\left(4m+1\right)\left(\psi^{(1)}\left(m+\frac{1}{2}\right)-\psi^{(1)}\left(m+1\right)\right)=\frac{2\Gamma\left(\frac{1}{4}\right)^{4}C}{\pi^{4}}\label{eq:cor3}
\end{equation}

\begin{align}
\sum_{m\geq0}\left[\frac{1}{4^{m}}\dbinom{2m}{m}\right]^{5}\left(-1\right)^{m}\left(4m+1\right) &=\frac{1}{8}\left(8\,_{5}F_{4}\left(\frac{1}{2},\frac{1}{2},\frac{1}{2},\frac{1}{2},\frac{1}{2};1,1,1,1;-1\right)\right. \nonumber \\
& \left.-\,_{5}F_{4}\left(\frac{3}{2},\frac{3}{2},\frac{3}{2},\frac{3}{2},\frac{3}{2};2,2,2,2;-1\right)\right) \nonumber \\
&=\frac{\Gamma\left(\frac{1}{4}\right)^{4}}{2\pi^{4}}\label{eq:cor4}
\end{align}

where $C$ is the Catalan's constant and $\psi^{(1)}(x)$ is the trigamma
function.
\end{cor}

\begin{proof}
Equations (\ref{eq:cor1}) and (\ref{eq:cor1.5}) follow from the $\nu=1/4$
and $\nu=3/4$ cases of (\ref{eq:cons1}) ; differentiating (\ref{eq:cor2}) with 
respect to $\nu$ before specializing to $\nu=1/2$, one arrives to
(\ref{eq:cor1}); twice differentiating (\ref{eq:cor1}) with respect to $\nu$ before specializing
$\nu=-1/2$, one arrives to (\ref{eq:cor3});
equation (\ref{eq:cor4}) follows from the $x=1/2$ case of (\ref{eq:cons2}). 
\end{proof}
Note that formula (\ref{eq:cons2}) allow us to evaluate quite easily
the moments of the function $K(x)K(1-x)$. 

\begin{cor}
\label{momentK}For every $n\in\mathbb{N}$ we have
\begin{align*}
 & \int_{0}^{1}x^{n}K\left(x\right)K\left(1-x\right)dx\\
 & =\frac{\pi^{3}\Gamma\left(n+1\right)^{2}}{8}\sum_{m\leq n/2}\left[\frac{1}{4^{m}}\dbinom{2m}{m}\right]^{4}\frac{4m+1}{\Gamma\left(n+2m+2\right)\Gamma\left(n+1-2m\right)}
\end{align*}
and for every $n\in\mathbb{N}^{+}$ we have
\begin{align*}
 & \int_{0}^{1}x^{n-1}K\left(x\right)\left(1-x\right)^{n-1}K\left(1-x\right)dx\\
 & =\frac{\pi^{7/2}}{8}\frac{\Gamma\left(n\right)^{3}\Gamma\left(n+\frac{1}{2}\right)}{\Gamma\left(2n\right)}\sum_{m<n}\left[\frac{1}{4^{m}}\dbinom{2m}{m}\right]^{4}\frac{4m+1}{\Gamma\left(\frac{1-2m}{2}\right)\Gamma\left(m+1\right)\Gamma\left(\frac{2m+1}{2}+n\right)\Gamma\left(-m+n\right)}.
\end{align*}
\end{cor}

\begin{proof}
Using (\ref{eq:cons2}) and switching the integral with the series
(it is quite easy to prove that it is allowed) we get
\[
\int_{0}^{1}x^{n}K\left(x\right)K\left(1-x\right)dx=\frac{\pi^{3}}{8}\sum_{m\geq0}\left[\frac{1}{4^{m}}\dbinom{2m}{m}\right]^{4}\left(4m+1\right)\int_{0}^{1}x^{n}P_{2m}\left(2x-1\right)dx
\]
and the result follows by the well-known identity
\[
\int_{0}^{1}x^{\mu-1}P_{n}\left(2x-1\right)dx=\frac{\Gamma\left(\mu\right)^{2}}{\Gamma\left(\mu+n+1\right)\Gamma\left(\mu-n\right)},\mathrm{Re}\left(\mu\right)>0
\]

(see \cite{GraRyz}, page $792$). Similarly, for the second identity
we use the relation
\begin{align*}
& \int_{0}^{1}x^{\mu-1}\left(1-x\right)^{\nu-1}P_{n}\left(2x-1\right)dx\\  & =\left(-1\right)^{n}\frac{\Gamma\left(\mu\right)\Gamma\left(\nu\right)}{\Gamma\left(\nu+\mu\right)}\,_{3}F_{2}\left(-n.n+1,\mu;1,\mu+\nu;1\right),\,\text{Re}(\mu)>0,\,\text{Re}(\nu)>0
\end{align*}
(see \cite{GraRyz}, page $792$) and the classical Watson theorem (see, for example, \cite{And}, Theorem $3.5.5$).
\[
_{3}F_{2}\left(a,b,c;\frac{a+b+1}{2},2c;1\right)=\frac{\sqrt{\pi}\Gamma\left(c+\frac{1}{2}\right)\Gamma\left(\frac{a+b+1}{2}\right)\Gamma\left(\frac{1-a-b}{2}+c\right)}{\Gamma\left(\frac{a+1}{2}\right)\Gamma\left(\frac{b+1}{2}\right)\Gamma\left(\frac{1-a}{2}+c\right)\Gamma\left(\frac{1-b}{2}+c\right)},
\]
with $\text{Re}(-a-b+2c)>-1$
\end{proof}

\section{Dougall's expansions, Mehler-Dirichlet theory and fl expasions}

There are interesting applications of Lemma $2.1$ in Zhou's paper, when we combine it with other identities.
We recall the Dougall's expansion (see \cite{Bat}, page $167$)

\begin{equation}
_{2}F_{1}\left(-\nu,\nu+1;1;x\right)=\sum_{m\geq0}\left[\frac{\sin\left(\pi\left(m-\nu\right)\right)}{\pi\left(m-\nu\right)}+\frac{\sin\left(\pi\left(m+\nu+1\right)\right)}{\pi\left(m+\nu+1\right)}\right]P_{m}\left(2x-1\right)\label{eq:FL 2f1}
\end{equation}
where the indeterminate forms must be interpreted as limits. Clearly,
this identity can be read as the FL expansion of the function $_{2}F_{1}\left(-\nu,\nu+1;1;x\right)$.
\begin{cor}
\label{cor:diff 2f1}We have that
\begin{align}
\sum_{m\geq0}\dbinom{2m}{m}\frac{\left(-1\right)^{m}}{4^{m}}\left[\frac{1}{\left(4m-1\right)^{2}}-\frac{1}{\left(4m+3\right)^{2}}\right] & =\frac{1}{9}\left(9\,_{3}F_{2}\left(-\frac{1}{4},-\frac{1}{4},\frac{1}{2};\frac{3}{4},\frac{3}{4};-1\right)\right. \nonumber\\
& \left.-\,_{3}F_{2}\left(\frac{1}{2},\frac{3}{4},\frac{3}{4};\frac{7}{4},\frac{7}{4};-1\right)\right) \nonumber \\
&=\frac{2\pi^{3/2}}{\Gamma\left(\frac{1}{4}\right)^{2}}.\label{eq:cor diff 2f1}
\end{align}
\end{cor}

\begin{proof}
From (\ref{eq:2f1 special}) and (\ref{eq:FL 2f1}) we get, taking
$x=1/2$ and recalling the well known relation
\[
P_{m}\left(0\right)=\begin{cases}
\tbinom{2m}{m}\frac{\left(-1\right)^{m}}{4^{m}}, & m\text{ even}\\
0, & m\text{ odd}
\end{cases}
\]
that
\[
\frac{\sqrt{\pi}}{\Gamma\left(\frac{1-\nu}{2}\right)\Gamma\left(\frac{\nu+2}{2}\right)}=\sum_{m\geq0}\dbinom{2m}{m}\frac{\left(-1\right)^{m}}{4^{m}}\left[\frac{\sin\left(\pi\left(2m-\nu\right)\right)}{\pi\left(2m-\nu\right)}+\frac{\sin\left(\pi\left(2m+\nu+1\right)\right)}{\pi\left(2m+\nu+1\right)}\right]
\]
and now the claim follows differentiating with respect $\nu$ both
sides and then taking $\nu=1/2.$
\end{proof}
Note that this result is interesting because, despite the seemingly
simple appearance, series like (\ref{eq:cor diff 2f1}) are often
linked to known, and important, mathematical constants but could be,
in general, difficult to deal with. An example is the series
\[
\sum_{m\geq0}\dbinom{2m}{m}\frac{\left(-1\right)^{m}}{4^{m}}\frac{1}{\left(4m+1\right)^{2}}
\]
which is closely related to the series
\[
\sum_{m\geq0}\dbinom{2m}{m}\frac{1}{4^{m}}\frac{H_{m}}{4m+1}
\]
and both are linked to lemnistate-like constants but, at present,
no technique is known for calculating their closed forms (see, for
more details on this topic, \cite{CamChu}).

Other interesting relations can be extrapolated from Zhou's paper;
indeed, from the well-known FL expansion
\[
K(x)=\sum_{m\geq0}\frac{2}{2m+1}P_{m}\left(2x-1\right),\,x\in[0,1),
\]
the Mehler-Dirichlet theory, the Hobson coupling formula, which states
that for $\nu\in\mathbb{C}$ and $\theta_{1},\theta_{2}\in\left[0,\pi\right)$
we have
\begin{align*}
&\frac{1}{2\pi}\int_{0}^{2\pi}P_{\nu}\left(\cos\left(\theta_{1}\right)\cos\left(\theta_{2}\right)+\sin\left(\theta_{2}\right)\sin\left(\theta_{2}\right)\cos\left(\phi\right)\right)d\phi\\ &=\begin{cases}
P_{\nu}\left(\cos\left(\theta_{1}\right)\right)P_{\nu}\left(\cos\left(\theta_{2}\right)\right), & \theta_{1}+\theta_{2}\leq\pi\\
P_{\nu}\left(-\cos\left(\theta_{1}\right)\right)P_{\nu}\left(-\cos\left(\theta_{2}\right)\right), & \theta_{1}+\theta_{2}\geq\pi
\end{cases}
\end{align*}
and its consequences (see Lemma 2.1 of \cite{zhou}), it is possible
to obtain the following ``quasi'' FL- expansions
\[
\sum_{m\geq0}\frac{P_{m}(2x-1)^{2}\left(-1\right)^{m}}{2m+1}=\frac{K(x)^{2}}{\pi},\,x\in\left[0,1/2\right],
\]
\[
\sum_{m\geq0}P_{m}(2x-1)^{2}z^{n}=\frac{2}{\pi}\frac{K\left(-\frac{16x\left(1-x\right)z}{(1-z)^{2}}\right)}{1-z},\,x,z\in\left(0,1\right)
\]
and then we are able to find the following identities:
\begin{cor}
We have that
\[
\frac{2}{\pi}\int_{0}^{1}\frac{K\left(\frac{16x\left(1-x\right)z^{2}}{(1+z^{2})^{2}}\right)}{1+z^{2}}dz=\begin{cases}
\frac{K(x)^{2}}{\pi}, & x\in\left[0,1/2\right]\\
\frac{K(1-x)^{2}}{\pi}, & x\in\left[1/2,1\right],
\end{cases}
\]
\[
\frac{2}{\pi}\int_{0}^{1}\frac{K\left(\frac{16x\left(1-x\right)z^{2}}{(1+z^{2})^{2}}\right)}{1+z^{2}}dx=\frac{\arctan(z)}{z},z\in(0,1).
\]
\end{cor}

\section{Conclusions}

We have shown some examples of how FL theory is a useful tool for
dealing with computational problems linked to some types of hypergeometric
functions and how its flexibility allows, at least in the first instance,
to be exploited in other areas of mathematics. We want to underline
how the results presented in this work are only a part of the possible
ones obtainable from the general formulas and how these techniques
could be used in other fields of mathematics and beyond; this lead
us to continue our investigation on these topics and we hope to produce
other interesting results in the future.

\section{Acknowledgments}

The author is a member of the Gruppo Nazionale per l'Analisi Matematica,
la Probabilit\`a e le loro Applicazioni (GNAMPA) of the Istituto Nazionale
di Alta Matematica (INdAM).


\begin{thebibliography}{10}

\bibitem{And} G. E. Andrews, R. Askey  and R. Roy, Special Functions, volume 
71 of Encyclopedia of Mathematics and Its Applications. Cambridge University Press, 
Cambridge, UK, 1999.

\bibitem{Bai}W. N. Bailey, Generalized Hypergeometric Series, Cambridge
University Press, Cambridge, 1935.

\bibitem{BarBerCha}N. D. Baruah, B. C. Berndt, H. H. Chan, Ramanujan\textquoteright s
series for $1/\pi$: a survey, Amer. Math. Monthly 116 (2009), 567--587

\bibitem{Bat}H. Bateman, Higher Transcendental Functions, volume
I, McGraw-Hill, New York, NY, 1953. (compiled by staff of the Bateman
Manuscript Project: Arthur Erd\'elyi, Wilhelm Magnus, Fritz Oberhettinger,
Francesco G. Tricomi, David Bertin, W. B. Fulks, A. R. Harvey, D.
L. Thomsen, Jr., Maria A. Weber and E. L. Whitney).

\bibitem{Bor}J. M. Borwein, R. E. Crandall, Closed forms: what they
are and why we care, Notices Amer Math Soc. 60 (1) (2013), 50--65.

\bibitem{Braf}F. Brafman, Generating functions of Jacobi and related
polynomials, Proc. of the American Math. Soc. 2 (6) (1951) 942--949.

\bibitem{BruKacPer}J. Br\"udern, J. Kaczorowski, and A. Perelli, Explicit
formulae for averages of Goldbach representations, Trans. Amer. Math.
Soc. 372 (2019), 6981--6999.

\bibitem{CamChu}J. M. Campbell, W. Chu, Lemniscate-like constants
and infinite series, accepted by Mathematica Slovaca.

\bibitem{CamDauSon}J.M. Campbell, J. D\textquoteright Aurizio, J.
Sondow. On the interplay among hypergeometric functions, complete
elliptic integrals, and Fourier--Legendre expansions, J. Math. Anal.
Appl., 479(1) (2019), 90--121.

\bibitem{CamCanDau}J. M. Campbell, M. Cantarini, J. D\textquoteright Aurizio,
Symbolic computations via Fourier--Legendre expansions and fractional
operators, accepted by Integral Transforms and Special Functions,
https://doi.org/10.1080/10652469.2021.1919103.

\bibitem{Can1}M. Cantarini, On the Ces\`aro average of the \textquoteleft Linnik
numbers\textquoteright . Acta Arith. 180(1) (2017), 45--62.

\bibitem{Can2}M. Cantarini, On the Ces\`aro average of the numbers
that can be written as sum of a prime and two squares of primes, Journal
of Number Theory 185 (2018),194--217.

\bibitem{Can3}M. Cantarini, Some identities involving the Ces\`aro
average of the Goldbach numbers, Math. Notes 106(5--6) (2019), 688--702.

\bibitem{CanDau}M. Cantarini, J. D\textquoteright Aurizio, On the
interplay between hypergeometric series, Fourier-Legendre expansions
and Euler sum, Bollettino Unione Matematica Italiana 12(4) (2019),
623--656.

\bibitem{CooGeYe}S. Cooper, J. Ge, D. Ye, Hypergeometric transformation
formulas of degrees 3, 7, 11 and 23, J. Math. Anal. Appl. 421(2) (2015),
1358--1376.

\bibitem{GraRyz}I. S. Gradshteyn and I. M. Ryzhik, Table of Integrals,
Series, and Products, edited by A. Jeffrey and D. Zwillinger, Academic
Press, New York, 7th edition, 2007.

\bibitem{Hob} E. W. Hobson, The Theory of Spherical and Ellipsoidal Harmonics, Cambridge University Press, 
Cambridge, UK, 1931.

\bibitem{KibMarSam}A. A. Kilbas, O. I. Marichev, S. G. Samko, Fractional
integrals and derivatives : theory and applications, Gordon and Breach
Science Publishers, Switzerland ; Philadelphia, Pa., USA, 1993.

\bibitem{Kris}G. Kristensson, Second Order Differential Equations,
Springer, New York, NY, 2010.

\bibitem{LanZac}A. Languasco, A. Zaccagnini, A Ces\`aro average of
Goldbach numbers, Forum Math. 27(4) (2015), 1945--1960.

\bibitem{zhou}Y. Zhou, Legendre functions, spherical rotations, and
multiple elliptic integrals. Ramanujan J. 34 (2014), 373--428. 
\end{thebibliography}
\end{document}